\documentclass[11pt, a4paper]{article}
\def\dj{d\kern-0.4em\char"16\kern-0.1em}

\usepackage{a4wide}

\usepackage{amssymb,amsmath,amsthm}
\usepackage{graphicx,enumerate}
\usepackage{ragged2e}
\usepackage{xcolor}
\usepackage[utf8]{inputenc}
\usepackage[T1]{fontenc}
\usepackage[english]{babel}
\usepackage{subcaption}
\usepackage{amssymb, amsmath, latexsym, amsthm, enumerate, amsfonts}
\usepackage[labelsep=period]{caption}

\usepackage[font=small]{caption}
\usepackage{mwe}

\usepackage[pdftex, pdfstartview=FitH]{hyperref}

\numberwithin{equation}{section}

 \newtheorem{thm}{Theorem}[section]

 \newtheorem{lem}[thm]{Lemma}
 \newtheorem{cl}[thm]{Claim}
 \newtheorem{obs}[thm]{Observation}
 
 \theoremstyle{definition}
 
 \theoremstyle{remark}

\newcommand\cC{{\mathcal C}}

\newcommand{\ex}{\mathop{}\!\mathrm{ex}}
\newcommand{\sat}{\mathop{}\!\mathrm{sat}}

\begin{document}

\title{Generalized saturation game}

\author{Bal\'azs Patk\'os\thanks{HUN-REN Alfr\'ed R\'enyi Institute of Mathematics, Budapest, 1053, Re\'altanoda utca 13-15, e-mail: \href{mailto:patkos@renyi.hu}{patkos@renyi.hu} Supported by NKFIH under grant FK132060.} \and Milo\v s Stojakovi\' c\thanks{Department of Mathematics and Informatics, Faculty of Sciences, University of Novi Sad, 21000 Novi Sad, Serbia, e-mail: \href{mailto:milosst@dmi.uns.ac.rs}{milosst@dmi.uns.ac.rs}. Partly supported by the Science Fund of the Republic of Serbia, \#7462, Graphs in Space and Time: Graph Embeddings for Machine Learning in Complex Dynamical Systems -- TIGRA. Partly supported by Ministry of Science,
Technological Development and Innovation of Republic of Serbia
(Grants 451-03-137/2025-03/200125 \& 451-03-136/2025-03/200125). Partly supported by Provincial Secretariat for Higher Education and Scientific Research, Province of Vojvodina (Grant No.~142-451-2686/2021).} \and
Jelena Stratijev\footnote{Department of Fundamental Sciences, Faculty of Technical Sciences, University of Novi Sad, 21000 Novi Sad, Serbia, e-mail: \href{mailto:jelenaknezevic@uns.ac.rs}{jelenaknezevic@uns.ac.rs}. This research has been supported by the Ministry of Science, Technological Development and Innovation (Contract No. 451-03-65/2024-03/200156) and the
Faculty of Technical Sciences, University of Novi Sad through project “Scientific and Artistic Research Work of Researchers in Teaching and Associate Positions at the Faculty of Technical Sciences, University of Novi Sad” (No. 01-3394/1).} \and M\'at\'e Vizer\thanks{Department of Computer Science and Information Theory, Budapest University of Technology and Economics, Budapest, Hungary. Supported by NKFIH grants FK 132060. \href{mailto:vizermate@gmail.com}{vizermate@gmail.com}} }

\date{}

\maketitle
\begin{abstract}
We study the following game version of the generalized graph Tur\'an problem. For two fixed graphs $F$ and $H$, two players, Max and Mini, alternately claim unclaimed edges of the complete graph $K_n$ such that the graph $G$ of the claimed edges must remain $F$-free throughout the game. The game ends when no further edges can be claimed, i.e.~when $G$ becomes $F$-saturated. The $H$-score of the game is the number of copies of $H$ in $G$. Max aims to maximize the $H$-score, while Mini wants to minimize it. The $H$-score of the game when both players play optimally is denoted by $s_1(n,\# H,F)$ when Max starts, and by $s_2(n,\# H,F)$ when Mini starts. We study these values for several natural choices of $F$ and $H$.
\end{abstract}

%\keywords{positional games, avoidance games, Sim, games on graphs}

%\msc{91A24, 05C57, 91A46}

\section{Introduction}

One of the earliest combinatorial games on graphs is Hajnal's triangle game in which two players alternatingly claim edges of the complete graph $K_n$ such that the graph $G$ of claimed edges cannot contain a triangle. The player who cannot move on his turn, loses. This innocent looking game turned out to be quite hard to analyze, so F\"uredi, Reimer, and Seress~\cite{Furedietal} introduced its quantitative version. The rules of the game remain the same, but the aims of the players are different. The \textit{score} of the game is the total number of edges claimed by the players during the game. One of the players, Max, targets to maximize the score, while the other player, Mini, aims to minimize the score. The \textit{game saturation number} is the score of the game when both players play according to an optimal strategy. Note that this might depend on who plays first, so precisely we have a Max-start and a Mini-start saturation number. Recent improvements on obtaining bounds on these numbers were obtained by Bir\'o, Horn, and Wildstrom \cite{biroetal}. Clearly, the rules and the scores of this game can be changed to forbid any graph $F$, or a family of graphs $\mathcal{F}$. The game in this generality was introduced in \cite{West} and results of this type were obtained in \cite{Englishetal, HKNS, LeeRiet, LeeRiet2, Spiro}. 

This game connects two basic problems of extremal graph theory. The graph of all claimed edges at the end of an $F$-saturation game is going to be \textit{$F$-saturated}, i.e.~$F$-free such that adding any non-edge would create a copy of $F$. The maximum number of edges in such an $n$-vertex graph is the \textit{Tur\'an number} $\ex(n,F)$, while the minimum number of edges in such an $n$-vertex graph is the \textit{saturation number} $\sat(n,F)$. By definition, the game saturation number of $F$ is between $\sat(n,F)$ and $\ex(n,F)$. Note that $\sat(n,F)$ is always at most linear in $n$ \cite{KT}, while $\ex(n,F)$ is only linear if $F$ is a forest, and quadratic if and only if $F$ is not bipartite.

After several sporadic examples, Alon and Shikelman~\cite{Alon} introduced the so-called \textit{generalized Tur\'an problem} that asks for the maximum number $\ex(n,\# H,F)$ of copies of $H$ that an $F$-saturated graph can have, and offered a methodical study of this parameter. Determining the parameter (and several of its variants) has become a popular topic in extremal graph theory in recent years. The generalized saturation number $\sat(n,\# H,F)$, the minimum number of copies of $H$ that an $n$-vertex $F$-saturated graph can have, was first defined in~\cite{Ketal} and results studying this value can also be found in~\cite{CL,ergetal,RV}.

In this paper, we study the game variant, first proposed in~\cite{CB}, of the above mentioned generalized problem. The \textit{$H$-score} of an $F$-saturation game is the total number of copies of $H$ at the end of the game (each subgraph isomorphic to $H$ is connected once). The $H$-score of the game when both players play optimally and Max starts is denoted by $s_1(n,\# H,F)$, and by $s_2(n,\# H,F)$ when Mini starts. When we have a conclusion that is valid independently of who starts the game, then we may simply write $s(n,\# H,F)$.

Our first results are about path-saturation games with star scores. By $P_s$ we denote the path on $s$ vertices, and $S_\ell$ denotes the star on $\ell$ vertices.

\begin{thm}\label{pathstars}
\
    \begin{enumerate}
        \item 
        If $s\ge 7$, then for any $\ell\ge 3$, we have $s(n,\# S_\ell,P_s)=\Theta_s(n^{\ell-1})$.
        \item 
        $s(n,\# S_3,P_4)=n-o(n)$, and $s(n,\# S_\ell,P_4)=0$ for all $\ell\ge 4$.
        \item 
        $s(n,\# S_3,P_5)=3n-o(n)$, $s(n,\# S_4,P_5)=n-o(n)$, $s(n,\# S_5,P_5)\le 1$ and \linebreak $s(n,\# S_\ell,P_5)=0$ for all $\ell\ge 6$.
       \item
       $s(n,\# S_3,P_6)=6n-o(n)$, $s(n,\# S_4,P_6)=4n-o(n)$, $s(n,\# S_5,P_6)=n-o(n)$ and $s(n,\# S_\ell,P_6)=0$ for all $\ell\ge 6$.
    \end{enumerate}
\end{thm}

Then we turn our attention to $S_4$-saturation games. This means that all vertices remain of degree at most 2 throughout the game, and at the end of the game all but at most one component of the game graph are cycles, and the possible exception is either an isolated edge or an isolated vertex. Therefore, we naturally consider only games with counting the number of paths or cycles. We summarize our findings in the following theorems.

\begin{thm}\label{t:p3}
When $n\leq 7$, the values of $s_1(n, \# P_3, S_4)$ and $s_2(n, \# P_3, S_4)$ are as given in Table~\ref{table:p3}.

For $n\geq 8$ and $i\in\{1,2\}$, we have
$$
s_i(n, \# P_3, S_4) = \begin{cases}
                        n, & \text{if } n \textrm{ and } i \textrm{ are of different parity,}\\ 
                        n-1, & \textrm{if } n \text{ is the same parity as } i.
                         \end{cases}
$$
\end{thm}

\begin{table}[htb]
\begin{center}
\begin{tabular}{c|ccccccc|ccccc}
 $n$ & 1& 2& 3& 4& 5& 6& 7& 8& 9& 10& 11& 12 \\ \hline
 $s_1$ & $n-1$& $n-2$& $n$& $n$& $n-1$& $n$& $n$& $n$& $n-1$& $n$& $n-1$& $n$ \\ 
 $s_2$ & $n-1$& $n-2$& $n$& $n$& $n$& $n-1$& $n$& $n-1$& $n$& $n-1$& $n$& $n-1$
\end{tabular}

\caption{Values of $s_1(n, \# P_3, S_4)$ and $s_2(n, \# P_3, S_4)$ for $n\leq 12$.}
\label{table:p3}
\end{center}
\end{table}

\begin{thm}\label{P5S3} For any $n \ge 1$ we have
$$\displaystyle{s(n,\#P_5,S_4)\leq 6}.$$ Additionally, $\displaystyle{s_2(n,\#P_5,S_4)\geq 5}$ for $n=4k$ and $k \geq 2$ or $n=4k+1$. 
\end{thm}

\begin{thm}\label{P6S3} For any $n \ge 1$ we have
$$\displaystyle{s(n,\#P_6,S_4)= 0}.$$ 
\end{thm}

Here we do not discuss the result for the score of the game where we count the number of $P_4$'s, as this result will be discussed later in Section~\ref{subs3.2}. 

Then we focus on cycle-free saturation games. Formally, we let $\cC=\{C_3,C_4,\dots\}$ be the set of all cycles. So at the end of the game, the game graph is a forest. We obtain the following results.

\begin{thm}\label{T2}
$\displaystyle{ \binom{\lfloor\frac{n}{2}\rfloor}{k-1} \leq s(n,\# S_k,\cC) \leq  \binom{\lceil\frac{n}{2}\rceil}{k-1}}$, $k > 4$.
\end{thm}

\begin{thm}\label{T3}
$\displaystyle{\frac{n^2}{16} + O(n) \leq s(n,\# P_4,\cC) \leq  \frac{n^2}{16} + O(n)}$.
\end{thm}

Finally, we address $P_5$-saturation games with scores being the number of triangles or paths of length $3$.

\begin{thm}\label{T1}
$\displaystyle{\frac{n-4}{3} \leq s(n,\#K_3,P_5) \leq \frac{n-4}{3} +4}$.
\end{thm}

\begin{thm}\label{p4p5}
    $s(n,\# P_4,P_5)=(3+o(1))n$.
\end{thm}

The rest of the paper is organized as follows. In Section \ref{ps}, we prove Theorem \ref{pathstars}. In Section \ref{deg2}, we consider $S_4$-saturated games and prove Theorem \ref{t:p3}. Section \ref{tree} addresses cycle-saturation games, while in Section \ref{p5} we consider $P_5$-saturation games. We end the paper with concluding remarks.

\section{Paths vs stars}\label{ps}

In this section, we prove Theorem \ref{pathstars}. First we show part (1). Observe that the number of $S_\ell$'s in a graph $G$ is $\sum_{v\in V(G)}\binom{\deg_G(v)}{\ell-1}$. Therefore, the upper bound $s(n,\#S_\ell,P_s)=O_{\ell,s}(n^{\ell-1})$ follows from the result of Erd\H os and Gallai \cite{EG} that states $\ex(n,P_s)\le \frac{s-2}{2}n$ and so $\sum_{v\in V(G)}\deg_G(v)\le (s-2)|V(G)|$ for any $P_s$-free graph $G$.

\medskip

To be able to formulate Max's strategy showing the lower bound in Theorem \ref{pathstars} (1), we need some definitions and some preliminary lemmas. We say that the component $C$ of a graph $G$ is \textit{Hamiltonian} if $|V(C)|\le 2$ or $C$ contains a Hamiltonian cycle.
For an $n$-vertex $F$-free graph $G_n$, in the Max/Mini-start $(G_n,F)$-game the two players claim edges of $E(K_n)\setminus E(G_n)$ alternately such that the claimed edges together with the edge set $E(G_n)$ stays $F$-free.

\begin{lem}\label{addhanging}
    Suppose $G_n$ contains a component $C$ and a vertex $v\in V(C)$ such that $C\setminus \{v\}$ has at least three components $C_1,C_2,C_3$ with $C_i\cup \{v\}$ containing a path of length $t$ that ends in $v$. Suppose the number of Hamiltonian components of $G_n$ of size at most $t$ is $T$. Then Max in the Mini-start $(G_n,P_j)$-game can achieve that $v$ has degree at least $T/3$ at the end of the game.
\end{lem}

\begin{proof}
Let $K^1,K^2,\dots,K^{m_i}$ be those Hamiltonian components of $G_n$ of size at most $t$ to which no adjacent edges have been played by either of the players before Max's move in round $i$. We claim that in his next move, Max can join any vertex $u$ of a largest $K^h$ to $v$. As Mini can touch at most two components per round, if we can prove the above claim, then the statement of the lemma follows.

We prove the following statement by induction on the number of rounds passed:  after Max's move, if $C^i$ denotes the component containing $v$ of the game graph after round $i$, then there are at least three components $C'$ of $C^i\setminus \{v\}$ such that $C'\cup \{v\}$ contains a path ending in $v$ that has strictly more vertices than any $K^j$. This is certainly true at the beginning of the game as this is the assumption of the lemma. So suppose the statement is true after round $i$. Then Mini can play an edge that keeps the game graph $P_j$-free and after her move, there are at least two components of $C^i \setminus \{v\}$ that together with $v$ contain a path ending in $v$ with strictly more vertices than a largest Hamiltonian component $K^h$. 

Now Max plays an edge $vu$ with $u\in K^h$, and we need to show that the game graph remains $P_j$-free. Suppose not. Then a copy of $P_j$ is created that uses $vu$ and possibly some part of a Hamiltonian path in $K^h$. Observe that $P_j$ can use at most one other component of $C^i\setminus \{v\}$, so the part of the copy of $P_j$  on the Hamiltonian path on $K^h$ can be replaced with the path on the non-used component ending in $v$. This is another copy of $P_j$ that was already present after Mini's move. Hence, Max can play as described.

That move of Max adds the third component containing a path ending in $v$, with strictly more vertices than any $K^j$. This concludes the inductive step, and the proof of the lemma.
\end{proof}

Let us extract the statement from the end of the proof of Lemma~\ref{addhanging}.

\begin{lem}\label{duplication}
If the current game graph $G$ contains a vertex $v$ and two induced Hamiltonian subgraphs $C_1,C_2$, all mutually vertex disjoint, such that $v$ is connected to both $C_1$ and $C_2$, 
%$G\setminus \{v\}$ has at least two Hamiltonian components $C_1,C_2$ with $|C_1|\ge|C_2|$, 
then joining a component $C_0$ not containing $v$, satisfying $|C_0|\le \min\{|C_1|,|C_2| \}$, to $v$ is a legal move.
\end{lem}

\begin{lem}\label{nicecomp}
Let $2\le S\le L<s$ be positive integers and $\varepsilon>0$ a positive real. Then in the $(E_n,P_s)$-game, where $E_n$ is an empty graph, Max can either achieve that there exists a vertex in the game graph with degree $\varepsilon \frac{n}{3}$, or that the game graph has $\frac{(1-\varepsilon)n}{S+L}$ Hamiltonian components, each of sizes $S$ and $L$.
\end{lem}

\begin{proof}
Max will complete his goal no later than when the game graph has at most $\varepsilon n$ isolated vertices. Max works in stages: in every odd stage, he builds a Hamiltonian component of size $S$ and in every even stage, he builds a Hamiltonian component of size $L$. We will show that he can build such a component claiming only edges in that component such that at the end of the stage there is at most one edge claimed by Mini that does not belong to any of the components of size $S$ or $L$. (We call these the \textit{nice components}.)

Eventually, we will prove that either Max is able to create $\frac{(1-\varepsilon)n}{S+L}$ nice components, each of sizes $S$ and $L$, or he is able to apply the strategy of Lemma \ref{addhanging} to obtain a vertex of degree at least $\varepsilon\frac{n}{3}$. First observe that if Mini ever claims an edge $uv$ with $v$ being a vertex of a nice component $C$ and $u\notin C$, then Max can claim an edge $vw$ with an isolated vertex $w$ and Lemma \ref{addhanging} can be applied with $t=1$ and Max creates a vertex of degree $\varepsilon\frac{n}{3}$. Therefore, from now on we can assume that this never happens.

Suppose Max has created $r$ nice components such that at most one edge claimed by Mini does not belong to these components. (The base case $r=0$ certainly holds at the beginning of the game.) Now Max wants to create another Hamiltonian component of size $M$, which is either $S$ or $L$ depending on the parity of $r$. If $M=2$, then if there is an edge $e$ played by Mini that does not belong to components created by Max, Max calls $e$ a nice component and stage $r+1$ is finished, while if there is no such edge, Max claims an isolated edge and calls that a nice component and stage $r+1$ is finished. 

Now we can assume $M\ge 3$. Suppose there exists an isolated edge $e$ played earlier by Mini that does not belong to nice components. Then Max claims an edge that is adjacent to $e$ and an isolated vertex, and from here on he always adds an isolated vertex $w$ to this current component $C$ such that $C$ contains a Hamiltonian path. If $C$ has at least 3 vertices, then Mini must claim an edge between $w$ and a vertex of $C$ as otherwise, Max claims an edge $xw'$, where $x$ is the unique neighbor of $w$, and $w'$ is an isolated vertex, and then Max can apply Lemma \ref{addhanging} with $v=x$, $t=1$. So, Max reaches a component of size $M$ with a Hamiltonian path, and in his next move he closes this path to a Hamiltonian cycle. Also, if $M=3$, then Mini has to close the $P_3$-component to a triangle as otherwise Max can apply Lemma \ref{addhanging}. This finishes stage $r+1$, possibly after Mini's move.

Finally, if after Max finishes a stage there are no isolated edges, then Max claims an isolated edge $e_1$, and Mini claims another isolated edge $e_2$, as otherwise Max creates a copy of $S_4$ and applies Lemma \ref{addhanging}. 

If $M=3$, then Max extends his edge to a copy of $P_3$ which Mini has to close into a triangle, as otherwise Max can create a copy of $S_4$, or a $P_l$, $l\in \{4,5\}$, with a hanging edge (depending on whether Mini touched $P_3$ or not), and apply Lemma \ref{addhanging}. So, if $M=3$, then stage $r+1$ is finished. 

If $M\ge 4$, then Max connects edges $e_1$ and $e_2$ to obtain a copy of $P_4$ that Mini next has to close to a $C_4$, as otherwise Max can connect an isolated vertex to the neighbor of a degree one vertex of the copy of $P_4$ and apply Lemma \ref{addhanging}. From here on, Max can follow the same strategy as above to obtain his Hamiltonian component of size $M$ and finish stage $r+1$.
\end{proof}

Now we are ready to prove the lower bound in Theorem \ref{pathstars} (1). The following theorem clearly implies the bound.

\begin{thm}\label{star}
    If $s\ge 7$, then in the $(E_n,P_{s})$-game Max can achieve that there exists a vertex in the game graph with degree $\Theta_s(n)$.
\end{thm}

\begin{proof}
We describe a strategy for Max that consists of three phases.

\smallskip

\textsc{Phase I} Building Hamiltonian components of size $L$ and $S$

\smallskip

In this first phase Max applies Lemma \ref{nicecomp} with $L$ and $S$ as below and $\varepsilon:=\frac{1}{18(S+L)}$. The value of $S$ and $L$ depends on $s$ as follows:
\begin{itemize}
\item
if $s=4k$, then $L=2k-1$, $S=2k-2$;
\item
if $s=4k+1$ or $s=4k+2$, then $L=2k$, $S=2k-1$;
\item
if $s=4k+3$, then $L=2k+1$, $S=2k$.
\end{itemize}
Note that at the end of this phase, the game graph either contains a vertex of high degree (in which case we are done) or at least $\frac{n}{2(S+L)}$ nice components, each of sizes $L$ and $S$, so the number of such components is at least 5 times the number of vertices that do not belong to nice components (with room to spare). We call these vertices \textit{bad}. In the remainder of the proof, w.l.o.g.~we assume the existence of the nice components and that $s\ge 11$ or $s=9$ which implies $3S\ge s$. We will deal with the cases $s=7,8,10$ separately.

\medskip

\textsc{Phase II} Forcing Mini to connect two nice components

\smallskip

At the beginning of Phase II, all components are Hamiltonian as nice components are such by definition, and the proof of Lemma \ref{nicecomp} ensures that all other components are isolated vertices with the exception of possibly one isolated edge. During this phase, components may merge. At any point of the game, a component that contains a bad vertex is a \textit{bad component}. If Max connects two nice components, then we call the resulting component a \textit{joined component}. A joined component $C$ becomes a \textit{problematic} component if Mini connects it to a bad component $C'$ with $|C|+|C'|\ge s$. A component $C$ is \textit{good} if no nice component can be connected to $C$ at any of its vertices. (Observe that a component can be good \textit{and} bad at the same time.) 

This phase terminates if Mini claims an edge that connects two Hamiltonian components of size at least $S$. Our aim is to show that Max has a strategy to terminate this phase while there are $\Theta_s(n)$ nice components of size both of $S$ and $L$ that are not connected to other components.

Max plays according to the following strategy. His first priority is that after his moves, all components should be either Hamiltonian or good. If Mini connects two Hamiltonian components, at least one of which is not an isolated vertex, then this merged component contains a Hamiltonian path, and so Max with his next move will make the union Hamiltonian again without increasing the length of the longest path and thus by a valid move. If Mini plays an edge within a Hamiltonian or a good component, then Max, if possible, responds by claiming an edge within any Hamiltonian or good component. If Mini connects two isolated vertices, the created component is Hamiltonian. Also, if two isolated vertices exist, then Max can connect them and still have all components Hamiltonian or good. 

Hence, Max cannot keep all components Hamiltonian or good only if all Hamiltonian components are cliques and good components are maximal. In this case, Max claims an edge connecting two nice components of \textit{odd} size (note that independently of $s$ exactly one of $L$ and $S$ is odd), and thus Max creates a joined component. Note that whenever Max creates a joined component, he joins two odd cliques. Also, note that by the assumption $3S\ge s$, a bad Hamiltonian component can contain at most 2 nice components.

Suppose there exists a joined component $C$ that is not yet good. Then Max plays as follows: he keeps his strategy of keeping all other components Hamiltonian or good with the modification that if Mini plays within a component, then Max claims an edge that makes $C$ Hamiltonian. If Mini connects a bad component $C'$ to a joined component $C$ making it bad, then if $|C|+|C'|< s$, it is valid to claim any edge within this component, and Max will claim an edge that will create a cycle containing all vertices of $C$. As $3S\ge s$ independently of $s$, no further nice components can be added to $C\cup C'$, thus this component becomes both good and bad and the number of components that are neither good nor bad (that are not nice) is zero again.

Next, suppose Mini connects the joined component $C$ to a bad Hamiltonian component $C'$ with $|C|+|C'|\ge s$. If that happens, we still refer to this newly formed component as $C$. Similarly, later in the proof, adding new components to $C$ will cause this component to grow. Let $C_1,C_2$ be the nice components joined in $C$. Let $uv$ be the edge joining them with $v\in C_1,u\in C_2$. As $C_1$ and $C_2$ are cliques (otherwise, as we already concluded, Max would not have been forced to create $C$) the edge $e$ connecting $C$ and $C'$ must contain $u$ or $v$, since otherwise $e$ would create a copy of $P_s$. Without loss of generality, we can assume $e=vw$ with $w\in C'$. 

If $|C'|\ge S$, then Max claims an edge $vz$ with $z\in C''$, a nice component of size $S$. This is a legal move by Lemma \ref{duplication}. %as a path containing vertices from $C''$ contains vertices from only one of $C_1\setminus \{v\},C_2,C'$, so the longest such path has $|C'|+1+S$ vertices, but paths with so many vertices existed in $C'\cup \{v\}\cup C_2$ which was a component after Mini's move. 
Now Max applies Lemma \ref{addhanging} with $t=S$ and $v$ to obtain a vertex of degree $\Theta_s(n)$. 

Similarly, if $|C_1|=|C_2|=L$, i.e.~$s=4k$ or $s=4k+3$, then the same move of Max lets him apply Lemma \ref{addhanging} even if $|C'|<S$ as then $C_1\setminus \{v\},C_2,C''$ are all of size at least $S$.

Now we can assume $s=4k+1$ or $s=4k+2$ and $|C'|<S$. Then, $|C_1|=|C_2|=S$, and Max claims an edge $vu_3$ with $u_3\in C_3$ a nice component of size $L$ (a legal move as the largest path goes from $C_3$ to $C_2$ and contains $L+1+S<s$ vertices). Now if Mini does not claim an edge connecting $C_2$ and $C_3$, then in his next move Max claims an edge $vu_4$ with $u_4\in C_4$ a nice component of size $S$ and applies Lemma \ref{addhanging} with $v$ and $t=S$ ($C_2,C_3,C_4$ are all of size at least $S$), to obtain a vertex of degree $\Theta_s(n)$. Therefore, assume that Mini claims an edge $e$ connecting $C_2$ and $C_3$. Observe that $e=uu_3$ as all other edges would yield a path on $L+S+S\ge s$ vertices (remember, all nice components are cliques by now). 

From here on Max will always claim an edge $u_hx$ with $x$ being one of $u,v$ or the previous $u_i$'s, and $u_h$ a vertex of a nice component $C_h$ of size $S$, where $h$ is always one larger than the previously used index. The fact that this is a legal move can always be seen by applying Lemma \ref{duplication}. Note that the set $U=\{u_1=v,u_2=u,u_3,u_4,\dots\}$ disconnects the $C_i$'s, so a longest path from $C_{i}$ to $C_j$ contains $|C_i|+|C_j|+a_{ij}$ vertices where $a_{ij}$ is the most number of inner vertices a path from $u_i$ to $u_j$ in $G[U]$ may contain. If Mini responds by claiming an edge $e$, then $e=u_hy$ with $y\in U$. Indeed, connecting $C_h$ to a vertex in $C_i\setminus U$  or a vertex from $U\setminus \{x\}$ to a vertex $C_h\setminus \{u_h\}$ would create a copy of $P_s$, while not connecting $C_h$ to any vertex of $U\setminus \{x\}$ allows Max to connect a $C_{h+1}$ to $x$ and apply Lemma \ref{addhanging} to obtain a linear degree vertex.

Let us take a closer look at how Max will connect new components.
Max first claims the edge $u_3u_4$, $u_4\in C_4$, where $C_4$ is a nice component of size $S$. If Mini connects $u_4$ to $u_1$ or $u_2$, then a copy of $P_{4k+1}$ is created from $C_4$ to $C_3$. So, if $s=4k+1$, this is not possible, and hence with his next move Max can claim $u_3u_5$ and apply Lemma \ref{addhanging} to obtain a vertex of linear degree.

If $s=4k+2$, then Mini has to claim either $u_4u_1$ or $u_4u_2$, say she claims $u_4u_2$. Then Max claims $u_5u_3$, and then Mini either does not connect $C_5$ to $U$ and lets Max claim $u_3u_6$ and obtain a linear degree vertex via Lemma \ref{addhanging} or her only possibility is to connect $u_5$ to $u_2$ as connecting to $u_1$ or $u_4$ would create a copy of $P_s$ from $C_5$ to $C_3$. But now the problematic component $C$ has become good, as connecting any further nice component to $C$ would create a copy of $P_s$.

\medskip

Let us now analyze how many nice components are merged to other components during this phase. As $3S\ge s$, any Hamiltonian component can contain at most 2 originally nice components. The way how good components are created shows that any good component can contain at most 5  originally nice components. How many times can Max be forced to join two odd nice components? As the number of edges between two such components is odd, therefore if the players only play these edges, then Max will claim the last one, and thus Mini will need to move next. Therefore, between two rounds when Max creates a joined component, either the number of bad components decreases or a joined component is connected to a bad component. We obtained that the total number of nice components that increase their size is at most six times the number of bad vertices at the beginning of the phase, so at most $\frac{n}{3(S+L)}$. Therefore, when Phase II finishes we still have at least $\frac{n}{6(S+L)}$ nice components, each of sizes $S$ and $L$.

\medskip

\textsc{Phase III}

\smallskip

We distinguish cases based on how Phase II terminates. 

\smallskip

\textsc{Case I} Mini connects two nice components at least one of which has size $L$.

Let $uv$ be the edge claimed by Mini connecting the two nice components with $v$ being in a component of size $L$. Then Max claims an edge $vw$ where $w$ is a vertex of a nice component of size $S$. This is a valid move as in the newly formed component the longest path contains at most $L+S<s$ vertices. Now, similarly to previous cases, Max can apply Lemma \ref{addhanging} with $v$ and $t=S$ to obtain a vertex of degree at least $\Theta_s(n)$.

\smallskip

\textsc{Case II} Mini connects two nice components of size $S$.

Let $u_1u_2$ be the edge claimed by Mini connecting the two nice components $C_1, C_2$ of size $S$. Then Max claims an edge $u_2u_3$ where $u_3$ is a vertex of a nice component of size $L$. This is a valid move as in the newly formed component the longest path contains  $L+S+1<s$ vertices. Note that this time we only know that nice components are Hamiltonian, but they are not necessarily cliques. However, in a Hamiltonian component of size $m$, between any pair of vertices, there exists a path on $\lceil \frac{m+2}{2}\rceil$ vertices. First observe that if Mini does not connect $C_3$ to $C_1$ or $C_2\setminus \{u_2\}$, then Max can claim an edge $u_2u_4$ with $u_4\in C_4$, $|C_4|=S$ (a legal move by Lemma \ref{duplication}!) and apply Lemma \ref{addhanging} with $v=u_2$ and $t=S$ to obtain a linear degree vertex. 

Next, observe that if Mini claims any edge $e$ with $e\neq u_3u_1$ that connects $C_3$ to $C_1$ or to $C_2\setminus \{u_2\}$, then a path on $L+\lceil \frac{S+2}{2}\rceil+S$ or $S+\lceil \frac{L+2}{2}\rceil+S$ vertices would be created. Considering the four cases of the possible residue class of $s$, one can easily check that $L+\lceil \frac{S+2}{2}\rceil+S\ge s$ and $S+\lceil \frac{L+2}{2}\rceil+S\ge s$ hold if $s\ge 9$.

Finally, if Mini claimed $u_3u_1$, Max claims an edge $u_3u_4$ with $u_4$ being a vertex in a nice component $C_4$, where the size of $C_4$ is $S$ if $s=4k+1$ or $s=4k+3$, while $|C_4|=L$ if $s=4k$ or $s=4k+2$. These are legal moves as the longest path is of size $2S+2$ in the former case, and $S+L+2$ in the latter case. Also, Mini next cannot claim any edge connecting $C_4$ to $C_1\cup C_2\cup C_3 \setminus \{u_3\}$ as that would create a path on at least $2L+2\ge s$ vertices if $s=4k$ or $4k+2$, and a path on $S+L+2\ge s$ vertices if $s=4k+1$ or $4k+3$. So in his next move, Max can claim an edge $u_3u_5$ with $u_5\in C_5$ a nice component of size $S$ and can apply Lemma \ref{addhanging} to obtain a vertex of linear degree.

\medskip

We now settle the remaining cases $s=7,8,10$. The proofs are all similar to the one above, so we do not elaborate on all the details. In all three cases, Max applies Lemma \ref{nicecomp} to build nice components of sizes 2 and 3 if $s=7$, of sizes 2, 3, and 4 if $s=8$ (formally, this requires a straightforward modification of the lemma), and of sizes 3 and 4 if $s=10$. In all cases, we set $\varepsilon=\frac{1}{100}$. Then Max tries to force Mini to connect two nice components. After creating the nice components there are $\frac{n}{100}$ bad vertices that are not in nice components. Max tries to keep every component Hamiltonian. A bad Hamiltonian component can contain at most three originally nice components. Max is only unable to follow this strategy if all components are cliques and it is his turn to move. Then he connects two nice triangles. Then three possibilities can happen: either the players take edges between the triangles alternately and thus Max passes the move to Mini, or Mini joins this joined component to another nice component and then Max can obtain a linear degree vertex using Lemma \ref{addhanging}, or Mini connects the joined component to a bad clique in which case one can see that either Max can finish with Lemma \ref{addhanging} or keep the component of bounded size. So altogether bad components can absorb at most 10 times their number of nice components and thus Mini will indeed be forced to connect two nice components. The below case analysis finishes the proof.

\textsc{Case A} $s=7$

%Max first applies Lemma \ref{nicecomp} with $S=2, L=3$, $\varepsilon=\frac{1}{100}$. 
If Mini connects a nice triangle $T$ to another nice component $C$ by claiming an edge $e=uw$ with $u\in T, w\in C$, then Max claims $uw'$ with $w'\in C'$ another nice isolated edge and applies Lemma \ref{addhanging} with $v=u$, $t=2$ to finish the proof. If Mini claims $u_1u_2$ with $u_1v_1,u_2v_2$ being nice edges, then Max claims an edge $u_1u$ with $u\in T$ a nice triangle. As $P_7$ is forbidden, within $T\cup \{u_1,v_1,u_2,v_2\}$ Mini can only claim $uu_2$ or an edge that is incident to $u_1$. If Mini does not claim $uu_2$ (including the possibility that she claims an edge not within this component), then Max in his next move claims $u_1u_3$ with $u_3$ in a third nice isolated edge and
 applies Lemma \ref{addhanging} with $v=u_1$, $t=2$ to finish the proof. If Mini claims $uu_2$, then Max claims an edge $uu_4$ with $u_4$ in another isolated edge and applies Lemma \ref{addhanging} with $v=u$, $t=2$ to finish the proof.

%All that remains to show is that Max can force Mini to connect two nice components (or two Hamiltonian components none of which is an isolated vertex). After creating the nice components by Lemma \ref{nicecomp}, Max tries to keep all components Hamiltonian. He cannot follow this strategy only if there exists at most one isolated vertex, all components are cliques, and it is his turn to move. Then he connects two triangles $T_1,T_2$ by claiming an edge $u_1u_2$ with $u_1\in T_1,u_2\in T_2$. If Mini makes $T_1\cup T_2$ Hamiltonian, then Max can keep claiming edges from between $T_1$ and $T_2$, and as the number of those edges is odd, Mini will have to play the first edge not there which either absorbs the last isolated vertex or connects two cliques and Max can finish as above. If Mini does not make $T_1\cup T_2$ Hamiltonian, then either $u_1$ or $u_2$ is a cut-vertex, say $u_1$ and Max can claim
%$u_1u_3$ with $u_3$ a vertex from a nice edge and finish the proof by applying Lemma \ref{addhanging} with $v=u_1$, $t=2$. This finishes the proof of this case.

\smallskip

 \textsc{Case B $s=8$}

 %This time Max applies Lemma \ref{nicecomp} in a little altered way. He builds nice components of three sizes: 2, 3, and 4 and sets $\varepsilon=\frac{1}{100}$. 
 If Mini claims an edge $u_1u_2$ with $u_1\in C_1,u_2\in C_2, |C_1|\ge 3$ and $C_i$ nice, then Max claims $u_1u_3$ with $u_3$ in a nice isolated pair and then applies Lemma \ref{addhanging} with $v=u_1$, $t=2$. If Mini claims an edge $u_1u_2$ with $u_1\in C_1,u_2\in C_2, |C_1|=|C_2|= 2$, then Max claims $u_1u_3$ with $u_3\in C_3$, $|C_3|=4$ a nice component. As the game is $P_8$-free, within $C_1\cup C_2\cup C_3$ Mini can play either $u_2u_3$ or an edge containing $u_1$. If Mini does not claim $u_2u_3$ in her next move, then Max claims $u_1u_4$ with $u_4$ in an isolated edge (a legal move by Lemma \ref{duplication}!) and applies Lemma \ref{addhanging} with $v=u_1$ and $t=2$. If Mini claims $u_2u_3$, then Max claims $u_3u_4$ with $u_4$ in an isolated edge and applies Lemma \ref{addhanging} with $v=u_3$ and $t=2$.

\smallskip

 \textsc{Case C $s=10$}

 If Mini claims an edge $u_1u_2$ with $u_1\in C_1,u_2\in C_2, |C_1|= 4$ and $C_i$ nice, then Max claims $u_1u_3$ with $u_3$ in a nice triangle and then applies Lemma \ref{addhanging} with $v=u_1$, $t=3$. If Mini claims an edge $u_1u_2$ with $u_1\in C_1,u_2\in C_2, |C_1|=|C_2|= 3$, then Max claims $u_1u_3$ with $u_3\in C_3$, $|C_3|=4$ a nice component such that within $C_3$ there is a Hamiltonian path from $u_3$ to any other vertices of $C_3$. There exists such a vertex as it can be read out from the proof of Lemma \ref{nicecomp} that $C_3$ has at least 5 edges on 4 vertices. As the game is $P_{10}$-free, within $C_1\cup C_2\cup C_3$ Mini can play either $u_2u_3$ or an edge containing $u_1$. If Mini does not claim $u_2u_3$ in her next move, then Max claims $u_1u_4$ with $u_4$ in a nice triangle (a legal move by Lemma \ref{duplication}!) and applies Lemma \ref{addhanging} with $v=u_1$ and $t=3$. If Mini claims $u_2u_3$, then Max claims $u_3u_4$ with $u_4$ in a nice triangle and applies Lemma \ref{addhanging} with $v=u_3$ and $t=3$.
 \end{proof}

Now we turn our attention to the proofs of (2), (3), and (4) of Theorem \ref{pathstars}. All lower bounds follow from the following simple lemma.

\begin{lem}\label{maxlinear}
  For any $2\le \ell\le s-1$, we have $s(n,\#S_\ell,P_{s})\ge (1+o(1))\binom{s-2}{\ell-1}n $.
\end{lem}

\begin{proof}
Max can apply Lemma \ref{nicecomp} with $S=L=s-1$ and $\varepsilon$ arbitrarily small to obtain pairwise vertex-disjoint Hamiltonian components of size $s-1$ that cover almost all vertices. These components cannot be joined either to each other or to other vertices as a copy of $P_s$ would be created, so at the end of the game, they should all be cliques yielding the lower bound.
\end{proof}

\begin{lem}\label{minlinear}
    For any $4\le s\le 6$, in the $P_s$-game Mini can keep all components of bounded size: 
    \begin{itemize}
        \item 
        in the $P_4$-free game, all components of size at most 3,
        \item 
        in the $P_5$-free game, all but one component of size at most 4, and the remaining component of size at most 5,
        \item 
        in the $P_6$-free game, all components of size at most 6.
    \end{itemize}
\end{lem}

\begin{proof}
   In the first phase of the strategy, Mini tries to keep all components of the game graph $H$ Hamiltonian. If Max joins two components $C_1$ and $C_2$ at least one of which is not an isolated vertex, then Mini can make the merged component Hamiltonian with a properly selected edge joining $C_1$ and $C_2$. Note, that as $C_1$, $C_2$ were Hamiltonian before Max's move, therefore after Max's move, the merged component contains a Hamiltonian path, so Mini's edge does not increase the length of the longest path, and thus it is legal in the game. If after Max's move, all components are Hamiltonian, then Mini can preserve this property, if there is a non-complete component or there are at least two isolated vertices. So Mini will be forced to move to a second phase of her strategy when the game graph is a union of pairwise vertex-disjoint cliques out of which at most one is an isolated vertex. A short case analysis depending on $s$ will finish the proof.

\smallskip

\textsc{Case I} $s=4$

If there is no isolated vertex in the game graph, then the game is over as joining two isolated edges would create a copy of $P_4$. If there is an isolated vertex, and an isolated edge, then Mini joins these, and Max can make this copy of $P_3$ a triangle, but then the game is over with all components at the end are either edges or triangles. A last subcase is when there is an isolated vertex, but no isolated edge in which case, the game is over with an induced triangle matching. 

\smallskip

\textsc{Case II} $s=5$

At the moment when Mini cannot follow her strategy from the first phase, the components of the game graph are $K_p$'s with $p\le 4$ with at most one isolated vertex. Mini will join two smallest such components. After this move, because of the $P_5$-free property, the only components that can be merged are isolated edges and possibly a single copy of $P_3$. Whenever Max connects two isolated edges, Mini will make that component a copy of $C_4$ thus making it unable to be connected to any other components. The only other option for Max is to connect an isolated edge $uv$ to the middle vertex $y$ of the $P_3$ $xyz$, say Max takes the edge $vy$. Then Mini takes the edge $uy$ and makes this component non-connectable to any other components. If later in the game, Mini has to join two isolated edges, then that copy of $P_4$ is not joinable to any further components. Therefore, at the end of the game, all components will have size at most 4 and at most one component of size 5.   

\smallskip

\textsc{Case III} $s=6$

At the moment when Mini cannot follow her strategy from the first phase, the components of the game graph are $K_p$'s with $p\le 5$ with at most one isolated vertex. Mini will join the two smallest such components. As two triangles cannot be joined, if there is at most one component of size smaller than 3 before Mini's move, then Mini joins this component to a larger one, and the final components are established, all of size at most 5. So Mini either joins an isolated vertex to an isolated edge to obtain a copy of $P_3$ or two isolated edges to obtain a copy of  $P_4$. We will call this component the \textit{current component}, and in general, there will be exactly one component, where Mini makes the first move in the second phase. Also, in this component, the players claim edges alternatingly, although not necessarily in a consecutive manner. Mini can achieve this, as Max may join two copies of $K_2$ or a copy of $K_2$ and a copy of $K_3$, but in this case, Mini makes this component Hamiltonian with 2 or 4 edges remaining to be claimed, and Mini only claims one such edge, if Max claims one. 

In the remainder of the proof, we present a short case analysis of how Mini can keep the current component of bounded size.

\smallskip

\textsc{Subcase III/A} With her first move in the second phase, Mini creates a copy of $P_4$ $uvwz$

If Max then plays an edge within this component, then Mini makes this component Hamiltonian, and as it is of size 4, it is not joinable to further components. When all edges are claimed, then Mini again joins two isolated edges. If Max joins the copy of $P_4$ to another component, then because of the $P_6$-free property, it must be an isolated edge $xy$ and the only possible connection, up to isomorphism, is the edge $xv$. Then Mini plays $xw$ and obtains a \textit{maximal} $P_6$-free component. 

\smallskip

\textsc{Subcase III/B} With her first move in the second phase, Mini creates a copy of $P_3$ $uvw$

If Max claims the edge $uw$, Mini returns to Subcase III/A. 

If Max joins a copy of $P_3$ to a triangle $xyz$, then the only way to join them (up to isomorphism) is by claiming the edge $xv$. Then, Mini can claim the edge $vy$ and after that, knowing that there are no further isolated vertices, the current component cannot be joined to further components. After that Mini gets back to Subcase III/A.

If Max joins a copy of $P_3$ to a $K_2$ such that the current component becomes a copy of $P_5$, then Mini closes this component to a copy of $C_5$, and the current component cannot be joined to further components. After claiming the remaining 5 edges of this component, Mini will be back to Subcase III/A.

Finally, suppose Max connects a copy of $P_3$ to an isolated edge $xy$ by claiming $xv$. Then  Mini plays $yu$. It can be checked that the current component cannot be joined to any further components, so after playing the remaining 5 edges of the component, Mini is back to Subcase~III/A.
\end{proof}

Now to get upper bounds in Theorem~\ref{pathstars}, parts 2, 3 and 4, we look at the cases within the proof of Lemma~\ref{minlinear}. We are interested in the set of components at the end of the game that gives the largest s. In the first case, that is an induced triangle matching, hence $s(n,\# S_3,P_4)\leq n-o(n)$. For the second case, that is a disjoint union of $K_4$'s, which gives $s(n,\# S_3,P_5)\leq 3n-o(n)$ and $s(n,\# S_4,P_5)\leq n-o(n)$. In the last case, the graph is a disjoint union of $K_p$'s, with $p \leq 5$, and possible a number of specially formed components on $6$ vertices. Hence, for the upper bound we can assume that at the end of the game we have almost a disjoint union of $K_5$'s, which gives $s(n,\# S_3,P_6)\leq 6n-o(n)$, $s(n,\# S_4,P_6)\leq 4n-o(n)$ and $s(n,\# S_5,P_6)\leq n-o(n)$. 

\section{$S_4$-saturation games}\label{deg2}

\subsection{Counting $P_3$'s}

The game graph at the end of the game is saturated and contains no copy of $S_4$, so its maximum degree is at most $2$. Hence, at most one of its connected components is not a cycle -- and it can be either an isolated edge or an isolated vertex. Counting copies of $P_3$ in such a graph with $n$ vertices, clearly, the number is either $n-2$, if there is an isolated edge, $n-1$, if there is an isolated vertex, or $n$, if all components are cycles.

To simplify the analysis of this game, we introduce some notation. We define the \emph{game deficit} as $s_i(n, \# P_3, S_4)-n$. As we concluded above, the game deficit can be either $-2$, $-1$, or $0$.

If we look at the game graph mid-game, the components that are cycles cannot be touched for the remainder of the game, and the number of copies of $P_3$ in them is equal to the number of vertices. Therefore, we can remove them from the graph for the rest of the game without influencing the deficit. 

Any connected component that is a path with at least three vertices is referred to as a \emph{long path}. Note that at any point in the game, any long path can be replaced with another (longer or shorter) long path without affecting the game deficit. The reason for that is that during the game long paths can be connected to other path-components, creating a new long path, or they can be closed into a cycle, in which case they can be removed from the game. For both of these actions, the length of the long path is irrelevant.

Therefore, looking at the game graph mid-play, the deficit of the game is uniquely determined just by knowing the following three parameters -- the number of isolated vertices $n^{(v)}$, the number of isolated edges $n^{(e)}$, and number of long paths $n^{(\ell)}$, and knowing whose turn it is. We denote such a game by $[n^{(v)}, n^{(e)}, n^{(\ell)}]_t$, where $t\in\{1,2\}$ depending on whether Max or Mini is next to play. With a slight abuse of notation, we will also use $[n^{(v)}, n^{(e)}, n^{(\ell)}]_t$ to denote the deficit of that game. Finally, if $n^{(\ell)} = 0$ we will omit it, writing $[n^{(v)}, n^{(e)}]_t$, and if both $n^{(e)}=n^{(\ell)} = 0$ we will omit them both, writing simply $[n^{(v)}]_t$. Hence, $[3,1,2]_1$ denotes the game with three isolated vertices, one isolated edge, and two long paths where Max starts, $[0, 3]_2$ denotes the game with three isolated edges on six vertices where Mini starts, and $[n]_i := s_i(n, \# P_3, S_4) - n$.

To prove Theorem~\ref{t:p3}, we need to determine $[n]_i$ for all $n$ and $i$. 

First, we present a simple but useful observation, stemming from the fact that the first player in their very first move of the game essentially has no choice -- all the potential moves are the same up to isomorphism.

\begin{obs}
    \label{o:first_move}
    $[n]_i = [n-2, 1]_{3-i}$, for every $n$ and $i$.
\end{obs}

We now handle the finite cases for $n\leq 7$.
\begin{lem} \label{l:p3_finite_cases}
The values of $[n]_i$, for $n\leq 7$ and $i\in\{1,2\}$, are as given in Table~\ref{table:p3}.
\end{lem}

\begin{proof}
    The cases $n\leq 3$ are trivial as for them the saturated graph is unique.

    For $n=4$ the deficit is 0. Max as first completes a copy of $P_4$ in his second move, and Max as second claims an isolated edge in his first move. Then in both cases, the only saturated supergraph is $C_4$.

    If $n=5$ and Max is first, he can complete a copy of $P_4$ in his second move guaranteeing $[5]_1 \geq -1$. On the other hand, Mini in her second move can complete a copy of $P_3$ and then, if Max does not close it into a 3-cycle, Mini closes the long path (with less than $5$ vertices) into a cycle in her following move, guaranteeing $[5]_1 \leq -1$.

    For $n=5$ with Max second, he claims an isolated edge in his first move and then completes a copy of $P_5$ in his second move, thus guaranteeing a zero deficit.

    If $n=6$, Max as first makes sure that after his second move one of the two cases holds: either there is a triangle, in which case the game is reduced to $[3]_2$, or there are three isolated edges, which he can complete into a copy of $P_6$ in his third move. In both cases there is a zero deficit.

    When $n=6$ and Mini is first, she can make sure that she completes a copy of $P_4$ in her second move. If Max does not close this path into a cycle (otherwise the game is reduced to $[2]_2$), there will be a long path after his move, so Mini can close the long path into a cycle (not spanning all the vertices), thus ensuring $[6]_2 \leq -1$. On the other hand, Max can complete a copy of $P_3$ in his first move, and unless Mini closes it into a triangle (in which case the game is reduced to $[3]_1$), Max can ensure a copy of $P_5$ after his second move. Thus $[6]_2 \geq -1$.

    For $n=7$, Max as first makes sure that after his second move one of the two cases holds: either there is a triangle, in which case the game is reduced to $[4]_2$, or there are three isolated edges. In the latter case Mini's next move creates a long path, that Max closes into a cycle in his following move, ensuring a zero deficit.

    If $n=7$ and Max is second, he can complete a copy of $P_3$ in his first move. Then he can make sure that after his second move, the graph is either a triangle and an isolated edge, which by Observation~\ref{o:first_move} reduces to $[4]_1$, or it is a copy of $C_4$, which reduces to $[3]_2$.
\end{proof}

We need to look at another three finite cases to complete a base for our general argument.

\begin{lem} \label{l:minimax-tree}
    $[12]_2, [11]_1, [8]_2 \leq -1$.
\end{lem}

\begin{proof}
    We show that in all three games Mini can guarantee a negative deficit, by applying the Minimax Algorithm (see, e.g.,~\cite{RN}). We need to demonstrate that for each of the three games, every possible move of Max has one response of Mini that leads to a game position with a negative deficit, recursively for the whole game tree. 

    Due to Observation~\ref{o:first_move}, it is enough to look at games $[10,1]_1$, $[9,1]_2$ and $[8]_2$. We build a game tree that covers all three games, see Figure~\ref{f:minimax}. 
    \begin{figure}[htb]
         \centering
        \includegraphics[width=0.9\textwidth]{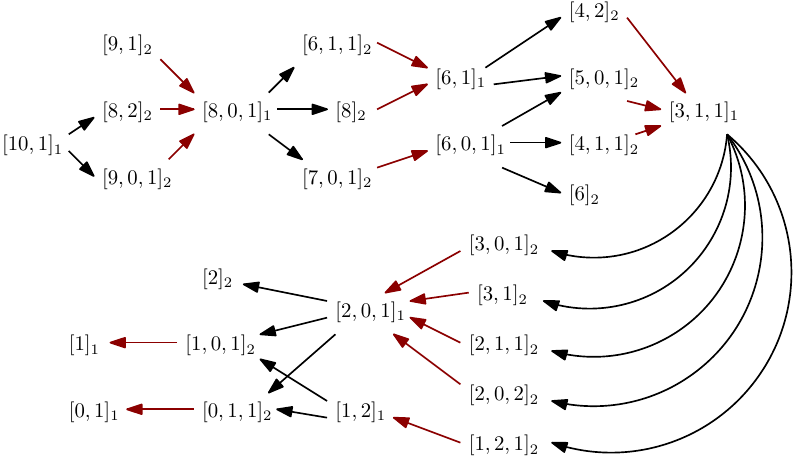}
        \caption{The Minimax game tree, giving an upper bound on the deficit for the three games in Lemma~\ref{l:minimax-tree}.}
        \label{f:minimax}
    \end{figure}
    The chosen moves of Mini that guarantee a negative deficit are denoted by red arrows. Note that in all the leaves of the game tree there are games with a negative deficit, either by Lemma~\ref{l:p3_finite_cases} and/or trivially. This concludes the proof.
    \end{proof}

\begin{lem} \label{l:max_path_extension}
    Let $n\geq 8$. If $n$ and $i$ have different parity, then $[n]_i = 0$. If $n$ and $i$ have the same parity, then $[n]_i \geq -1$.
\end{lem}

\begin{proof}
    From the beginning of the game, Max will follow the \emph{path extension strategy}, making sure that after each of his moves the played edges form a number of cycles, possibly none, and exactly one path of positive length. Let us first show that, as long as there are at least two isolated vertices in the graph, he can follow his strategy. Indeed, independent of who starts the game, Max in his very first move can create a single path, either a copy of $P_2$ or a copy of $P_3$. Then, if we assume that Max followed this strategy for a number of rounds, in her following move Mini can do one of three things: close a cycle, in which case Max claims an isolated edge, extend the path, in which case Max extends the path, or claim an isolated edge, in which case Max joins the isolated edge to the path, thus again extending it.

    Max follows the path extension strategy until the number of isolated vertices drops below $7$. If $n$ and $i$ originally had different parity, then there are exactly six isolated vertices at that point, and the game is reduced to either $[6,1,0]_2$ or $[6,0,1]_2$. In $[6,1,0]_2$ Max can guarantee that after his move there is one copy of $P_4$ and $4$ isolated vertices. Whatever Mini plays, Max can reduce the game to either $[3]_2$ or $[2,1]_2 = [4]_1$, by Observation~\ref{o:first_move}. In both cases the game deficit is zero. In $[6,0,1]_2$ Mini has several options. She can close the long path into a cycle, reducing the game to $[6]_1$. Alternatively, she can claim an isolated edge, to which Max can respond by closing the long path into a cycle, reducing the game to $[4,1,0]_2 = [6]_1$, by Observation~\ref{o:first_move}. Finally, she can add one edge to the long path, to which Max can again respond by closing the long path into a cycle, reducing the game to $[5]_2$. In all cases the game deficit is zero.

    If $n$ and $i$ originally had the same parity, when Max stops following the path extension strategy there are exactly five isolated vertices, and the game is reduced to either $[5,1,0]_2$ or $[5,0,1]_2$. By Observation~\ref{o:first_move}, the deficit of the first one is $[7]_1$, which is zero. In $[5,0,1]_2$, Mini has three options. She can close the path into a cycle, reducing the game to $[5]_1$. Or she can extend the path, in which case Max closes that path and we have $[4]_2$. Finally, if Mini plays an isolated edge, Max closes the path into a cycle reducing the game to $[3,1,0]_2 = [5]_1$, by Observation~\ref{o:first_move}. For all three options, the game deficit is at least $-1$.
\end{proof}

\begin{lem}\label{l:reduction_by_345}
    Let $\ell\geq 3$. We have $[2\ell]_2 \leq \max\{[2\ell -4]_2, [2\ell -5]_1 \}$ and $[2\ell +1]_1 \leq \max\{[2\ell -2]_2, [2\ell -3]_1 \}$.
\end{lem}

\begin{proof}
    If $n=2\ell$ and Mini is first to play, she makes sure to complete a copy of $P_4$ after her second move. Then Max can play the following three moves. If he closes the path into a cycle, the game reduces to $[2\ell -4]_2$. If he claims an isolated edge, Mini closes the copy of $P_4$ into a cycle, thus reducing the game to $[2\ell -6,1]_1 = [2\ell -4]_2$, by Observation~\ref{o:first_move}. Finally, if he extends the path into $P_5$, Mini closes it into a cycle reducing the game to $[2\ell -5]_1$.

    If $n=2\ell+1$ and Mini is second to play, she completes a copy of $P_3$ in her first move. Then similarly to the previous case, Max can play the following three moves. If he closes the path into a cycle, the game reduces to $[2\ell -2]_2$. If he claims an isolated edge, Mini closes the copy of $P_3$ into a cycle, thus reducing the game to $[2\ell -4,1]_1 = [2\ell -2]_2$. Finally, if he extends the path into $P_4$, Mini closes it into a cycle reducing the game to $[2\ell -3]_1$.
\end{proof}

We are ready to wrap up the proof of the main theorem.

\begin{proof}[Proof of Theorem~\ref{t:p3}]
All the cases $n\leq 7$ are covered by Lemma~\ref{l:p3_finite_cases}.

Furthermore, for all the cases where $n\geq 8$ and $n$ and $i$ are of the same parity, we have that $[n]_i \leq -1$, either by Lemma~\ref{l:minimax-tree} or by recursively applying Lemma~\ref{l:reduction_by_345}. When this conclusion is coupled with the second part of Lemma~\ref{l:max_path_extension} we get $[n]_i = -1$ in that case.

Finally, when $n\geq 8$ and $n$ and $i$ are of different parity, we have $[n]_i = 0$ by the first part of Lemma~\ref{l:max_path_extension}. This concludes the proof of the theorem.
\end{proof}

\subsection{Counting $P_4$'s}\label{subs3.2}

The case of counting $P_4$'s is in many ways similar to the one of counting $P_3$'s in the same setting, and much of our setup and approach from the previous section can be used here as well. Having this in mind, along with the fact that we do not aim at determining the exact values of the score for all $n$ but rather some non-trivial upper and lower bounds, we will present proofs in a more condensed fashion.

Let us first observe that in a saturated $S_4$-free graph the numbers of copies of $P_3$ and $P_4$ are the same on all connected components except on triangles, where each such component has three copies of $P_3$ and zero copies of $P_4$. One immediate consequence follows.
\begin{obs} \label{o:p4-less-than-p3}
    $s_i(n, \# P_4, S_4) \leq s_i(n, \# P_3, S_4)$.
\end{obs}

It will again be useful to define the \emph{deficit of the game} as $s_i(n, \# P_4, S_4)-n$. We will show that this value is between $-1$ and $-3$, except for several small values of $n$. Also, we use a short notation $[n]_i$ to denote the game on $n$ vertices, with $i$ being $1$ or $2$ depending if Max or Mini is first to play. With slight abuse of notation, we also use $[n]_i$ to denote the game deficit under optimal play.

Clearly, if during the game there is a cycle with at least four edges it can be removed from the game, along with its vertices, without affecting the deficit in the rest of the game.

We first handle some cases of small $n$.
\begin{lem} \label{l:p4-finite-cases}
    For $n\leq 3$ we have $[n]_i = -n$. Also, $[4]_1 = [4]_2 = [5]_2=0$, $[5]_1\geq-1$, and $[8]_1=[9]_2 \leq -2$.
\end{lem}
\begin{proof}
    The cases $n\leq 3$ are trivial. 
    
    On $4$ vertices, depending on who starts, Max can create either a copy of $P_4$ or two disjoint edges, and in both cases, the only saturated supergraph is a copy of $C_4$. 
    
    On $5$ vertices, Max as second claims an isolated edge in his first move, and then creates a copy of $P_5$ in his second move. Max as first in his second move creates a copy of $P_4$ and thus ensures that the deficit is at least $-1$.

    In the game $[8]_1$ Mini right away creates a copy of $P_3$, and Max has to extend it to avoid the creation of a triangle. Then Mini further extends it, getting to three isolated vertices and a long path. From there, all possible moves result in a deficit of at least $-2$ as Mini will close the long path to a cycle.

    In the game $[9]_2$, Max has to claim an isolated edge in his first move to avoid the creation of a triangle right away. If Mini then extends that path into a copy of $P_3$, Max has to play adjacent to it. In any case Mini can complete a copy of $P_6$ in her next move, and make sure that the long path is closed into a cycle in the round that follows. 
\end{proof}

We now show that Max can prevent the deficit from going below $-3$.
\begin{lem} \label{l:p4-lower-bd}
    We have $[n]_i \geq -3$, for all $n$ and $i$.
\end{lem}

\begin{proof}
    If Max is first, he follows the path extension strategy similar to the one presented in the proof of Lemma~\ref{l:max_path_extension}. If Mini ever decides to close the path into a cycle, it will be an even cycle and thus not a triangle. Max keeps following the path extension strategy as long as there are isolated vertices, and that clearly ensures a deficit of at least $-3$.

    If Max is second, he claims an isolated edge in his first move. Then after his second move, he can make sure that the played edges form either a copy of $P_5$ or a $P_4 + P_2$. In the former case, he applies the path extension strategy, and if Mini ever completes a cycle Max proceeds as first on the remaining isolated vertices. In the latter case, Mini can either create a copy of $C_4$, or Max can respond by creating a $C_4 + P_2 + P_2$, or Max can respond by creating a copy of $P_7$. If a copy of $P_7$ is created, Max simply follows the path extension strategy similarly as above, and if instead $C_4 + P_2 + P_2$ is created, one can ignore $C_4$ and hence proceed inductively back to the initial condition of this case.
 
\end{proof}

\begin{lem} \label{l:mini-extends-matching}
    For $\ell\geq 0$, we have $[4\ell+3]_2 = -3$.
\end{lem}

\begin{proof}
    We present a strategy for Mini that will ensure the creation of a triangle. Together with Lemma~\ref{l:p4-lower-bd}, this will complete the proof.

    Throughout the game, Mini extends the matching of the game graph by playing isolated edges. Note that to avoid the creation of a triangle, Max cannot connect an isolated edge with an isolated vertex. If he ever connects two isolated edges to create a copy of $P_4$, Mini closes it into a copy of $C_4$ and removes it from the game. Due to parity, towards the very end of the game there will be only one isolated edge and one isolated vertex, with Max to play, and the triangle will eventually be created.
\end{proof}

\begin{lem} \label{l:p4-recursion-mini}
    We have $[n]_1\leq [n-4]_1$, and $[n]_2\leq \max \{ [n-4]_2, [n-5]_1\}$.
\end{lem}

\begin{proof}
    Mini as second can create a copy of $P_3$, and then Max has to extend it to a copy of $P_4$ to avoid the creation of a triangle. Mini then closes it into a copy of $C_4$ and the game is reduced to $[n-4]_1$.

    If Mini is first she can create a copy of $P_4$ after two moves. Then if Max closes it into a $C_4$ or claims an isolated edge, Mini responds by doing the other one of those two things reducing the game to $[n-4]_2$. If, however, Max extends the copy of $P_4$ into a $P_5$, Mini closes it into a cycle reducing the game to $[n-5]_1$.
\end{proof}

Putting together the results from Observation~\ref{o:p4-less-than-p3}, Lemma~\ref{l:p4-finite-cases}, Lemma~\ref{l:p4-lower-bd}, Lemma~\ref{l:mini-extends-matching} and Lemma~\ref{l:p4-recursion-mini}, this is what we know about the score of the game.

\begin{thm} \label{t:p4}
    The values of $s_i(n, \# P_4, S_4)$ are as given in Table~\ref{table:p4}, with universal lower bound of $n-3$ everywhere.
\end{thm}

\begin{table}[h]
\begin{center}
\begin{tabular}{c|ccccc|cccc}
 $n$ & 1& 2& 3& 4& 5& 4k+6& 4k+7& 4k+8& 4k+9 \\ \hline
 $s_1$ & $n-1$& $n-2$& $n-3$& $n$& $n-1$& $\leq n-2$& $n-3$& $\leq n-2$& $\leq n-1$ \\ 
 $s_2$ & $n-1$& $n-2$& $n-3$& $n$& $n$& $\leq n-1$& $n-3$& $\leq n-1$& $\leq n-2$
\end{tabular}

\caption{Values of $s_1(n, \# P_4, S_4)$ and $s_2(n, \# P_4, S_4)$, for $k\geq 0$, with universal lower bound of $n-3$ everywhere in the table.}
\label{table:p4}
\end{center}
\end{table}

\subsection{Counting $P_5$'s}

\begin{proof}[Proof of Theorem~\ref{P5S3}]
%In this game, Max desires big cycles $C_l$ where $l\geq 5$, while Mini wants a disjoint union of triangles and $C_4$s. 

We give a strategy for Mini, without specifying the first player.

\medskip 

{\bf Stage 1:} There is at least one isolated vertex. 

Let $S$ be the set of all nontrivial connected components that are not cycles, before a move of Mini. Before each of her moves, Mini chooses the first of the following three rules whose assumption is satisfied:
\begin{enumerate}
    \item\label{P5S3S11} $S$ contains a copy of $P_i$, where $i \in \{3,4\}$. Mini 'closes' such a path to a copy of $C_i$.
    \item\label{P5S3S12} $S$ consists of edges. Mini tries to claim an isolated $K_2$. If that is not possible, Mini connects one edge with an isolated vertex creating a copy of $P_3$.
    \item\label{P5S3S13} $S$ is an empty set. Mini plays an isolated $K_2$ if she can. If she can not, then the graph of the played edges consists of cycles and the game is over. 
\end{enumerate}

{\bf Stage 2:} There is no isolated vertex.

At this stage, Mini chooses one of the following options depending on the type of connected components 
%of the graph of the played edges 
beside the cycles and $K_2$'s.
\begin{enumerate}
    \item\label{P5S3S21} One copy of $P_3$ and one copy of $P_4$. Mini closes a copy of $C_4$. Note that after Max's next move, the game is continued in one of the following cases: \ref{P5S3S21}, \ref{P5S3S24}, or \ref{P5S3S25}.
    \item\label{P5S3S22} One copy of $P_3$. Mini closes a triangle. Then until the end of the game, she just closes copies of $C_4$. 
    \item\label{P5S3S23} One copy of $P_4$. Mini closes it. Then until the end of the game, she just closes copies of $C_4$.
    \item\label{P5S3S24} One copy of $P_5$. Mini closes a copy of $C_5$. Note that after Max's next move, the game is continued in: \ref{P5S3S23}.
    \item\label{P5S3S25} None. Mini connects two isolated edges and completes a copy of $P_4$ and then until the end of the game, she always closes cycles. If there are two paths, she closes one arbitrarily.
\end{enumerate}

Now we prove that Mini can follow her strategy. It is clear that she can follow rules \ref{P5S3S11} and \ref{P5S3S12} from Stage 1. 
If it is Mini's turn to play by rule \ref{P5S3S13} and there is just one isolated vertex, all the remaining components are cycles, so the game is over. Therefore, it is clear that Mini can follow Stage 1.

Note that since the length of the cycles that can be created in Stage 1 are at most 4, the score of the game is 0, if the game is finished in Stage 1.
\smallskip 

Observe that after Mini's last move played in Stage 1 the graph of the played edges is a disjoint union of the following components, some copies of $C_3$, some copies of $C_4$, some copies of $K_2$, and at most one copy of $P_3$. Since a cycle cannot be connected to any other component, all possible scenarios before the Mini's first move in Stage 2 are listed in the five cases above.

If Mini applies one of the rules \ref{P5S3S22} or \ref{P5S3S23}, it is clear that she can close the cycle and after that, only the pairs of isolated edges can be connected into one component, hence the only possible move for Max is to complete a copy of $P_4$, and then Mini closes a copy of $C_4$ and these moves will be repeated to the end of the game. Here the graph at the end of the game does not contain a copy of $P_5$, hence the score is $s=0$.

If Mini applies rule \ref{P5S3S24}, the score is $s=5$.
If Mini applies rule \ref{P5S3S25}, that is the only possible move for her and after that, Max can complete another copy of $P_4$ component or create a copy of $P_6$. In both of these cases, Mini will close a cycle in her following move. If she closes a copy of $C_6$ (that is possible just once), the game is reduced to \ref{P5S3S23}. Otherwise, Mini closes a copy of $C_4$ and repeats closing a cycle such that at the end of the game the graph contains at most one copy of $C_6$ and all the remaining cycles have at most $4$ vertices. Therefore, the score of the game is $s \leq 6$.
This concludes the first part of the theorem.
\medskip 

Now we look at Max's strategy when he is the second player and $n=4k$ or $n=4k+1$. The game stops at the moment when the first cycle on at least $5$ vertices is made, showing that the score is at least $5$. 

While there is no cycle in the graph:
\begin{enumerate}
    \item If there is a $P_i$ component, where $i \geq 5$, Max closes a copy of $C_i$ and the game is over.
    \item If Mini plays an isolated $K_2$, Max does the same (because of the form of $n$).
    \item\label{P5S33} Else, if Mini plays in a component that is a copy of $P_i$ ($i \ge 3$) after her move, Max adds one isolated edge to this path, making a copy of $P_{i+2}$, or if there is no isolated edge he adds an isolated vertex, making a copy of $P_{i+1}$.
\end{enumerate} 

When Max has to apply rule \ref{P5S33} for the first time, Mini has completed either a copy of $P_3$ or a copy of $P_4$. In case this component is a copy of $P_3$, there has to be at least one isolated edge because the first round has been finished with two isolated edges. Otherwise, if the component is a copy of $P_4$ there must be either an isolated edge or vertex. Therefore, after this move of Max, there must be a $P_i$ component, where $i \geq 5$, so the game will be over in the next round with $s_2 \geq 5$.

If the graph is matching and there is at most one isolated vertex, it must be Mini's turn (because of the condition for $n$). Therefore, Mini has to complete either a copy of $P_3$ or a copy of $P_4$, and in his following move, Max completes a copy of $P_i$ where $i \geq 5$, following rule \ref{P5S33}, so the game will be finished in the next round with $s_2 \geq 5$.
%(We omit the part of the proof that $\displaystyle{s_1(n,P_5,S_4)\geq 5}$ for $n=4k+2$).
\end{proof}

\subsection{Counting $P_6$'s}

\begin{proof}[Proof of Theorem~\ref{P6S3}]
We give a strategy for Mini, independently of who is the first player.

We look at the graph of already played edges before a move of Mini and check the set $S$ of nontrivial connected components that are not cycles. Before each of her moves, Mini chooses the first of the following four rules, where the assumption is satisfied:
\begin{enumerate}
    \item $S$ contains a copy of $P_i$, where $i \in \{3,4,5\}$. Then Mini closes it to a copy of $C_i$.
    \item $S$ consists of two independent edges. Then Mini completes a copy of $P_4$, connecting these isolated edges. 
    \item\label{P6S33} $S$ consists of a single edge. Mini completes a copy of $P_3$, connecting one isolated vertex to $K_2$.
    \item\label{P6S34} $S$ is an empty set. Mini plays an isolated $K_2$. 
\end{enumerate}

The only possible cases where Mini cannot answer using her strategy are rules \ref{P6S33} and \ref{P6S34}, but then the game is over before that move because either $S=K_2$ and there is no isolated vertex, or $S= \emptyset$ and there is at most one isolated vertex and all the remaining connected components are cycles on $3$, $4$ or $5$ vertices. Therefore, when Mini cannot follow this strategy the game is over. 

At the moment when Mini has finished her move, the graph contains, some copies of $C_3$, some copies of $C_4$, some copies of $C_5$, and at most one component from the set $\{K_2, P_3, P_4\}$. Therefore, after Max's move the set $S$ is exactly one of the $\{\emptyset, K_2, 2K_2, P_3, P_4, P_5, P_3+K_2, P_4+K_2\}$. Following her strategy, Mini ensures that there is no cycle on more than $5$ vertices in the graph. Hence, $\displaystyle{s(n,\#P_6,S_4)= 0}$. 
\end{proof}

\section{$\cC$-saturation games}\label{tree}

In this section, we consider games with the set of all cycles $\cC=\{C_3,C_4,\dots\}$ being forbidden. Note that in any such game, there are exactly $n-1$ claimed edges during the game, and the resulting game graph is a tree. We start with a simple lemma. 

\begin{lem}\label{lemaTn}
Let $T$ be a tree on $\lfloor\frac{n}{2}\rfloor+1$ vertices. If the players alternately
claim edges of $K_n$ and neither is allowed to play a cycle, a player can build $T$. 
\end{lem}
\begin{proof}
During the game, a player that we call TreeBuilder builds greedily a connected component $C$ such that at the end of the game, this component will contain a tree isomorphic to $T$. Let $V(T)=\{v_1,v_2,\dots,v_{|V(T)|}\}$ such that for any $1\le j \le |V(T)|$ the graph $T_j:=T[v_1,\dots,v_j]$ is a tree. Suppose TreeBuilder managed to build a tree isomorphic to $T_j$ in $C$, for some $j$, and he is to move next and the game is not yet over. Then there exists a component $C'\neq C$, and TreeBuilder claims an $u'u$, where $u'$ is an arbitrary vertex of $C'$, and $u$ is the vertex that plays the role of the only neighbor of $v_{j+1}$ in $T_{j+1}$.

If TreeBuilder starts the game, then his first claimed edge creates $T_2$ and he will play $\lceil \frac{n-1}{2}\rceil-1$ more moves, and thus can build any tree on $2+\lceil \frac{n-1}{2}\rceil-1=\lfloor \frac{n}{2}\rfloor+1$ vertices. If TreeBuilder's opponent starts the game, then the opponent's first edge builds $T_2$ and TreeBuilder still has $\lfloor \frac{n-1}{2}\rfloor$ moves yielding the same bound.
\end{proof}

\begin{lem}\label{numberstar}
    If a tree $T$ on $n$ vertices contains at least $x$ vertices of degree at least 2, then the number of copies of $S_k$ in $T$ is at most $\binom{n-x}{k-1}$, provided $k\ge 4$.
\end{lem}

\begin{proof}
    The number of copies of $S_k$ is $\sum_{i=1}^n\binom{d_i}{k-1}$, where $d_i$ is the degree of vertex $v_i\in V(T)$. For any $a\le b$, we have $\binom{a}{k-1}+\binom{b}{k-1}\le \binom{a-1}{k-1}+\binom{b+1}{k-1}$. Therefore, under the conditions of the lemma, $\sum_{i=1}^n\binom{d_i}{k-1}$ is maximal if $x-1$ vertices have degree 2, $n-x$ vertices have degree 1, and one vertex has degree $n-x$. This is realizable by a $P_{x+2}$ with all vertices not on the path connected to the same vertex of the path that is not end-vertex.
\end{proof}

The \textit{double star} $D_{x,y}$ is the tree on $x+y+2$ vertices with $x+y$ leaves and a vertex of degree $x+1$ and a vertex of degree $y+1$ (the \textit{centers} of the double star). If $|x-y|\le 1$, then we say the double star is \textit{balanced}.

\begin{lem}\label{doublestar}
    If a tree $T \neq D_{x,y}$ on $n$ vertices contains a double star $D_{x,y}$, then $T$ contains at least $xy+\min\{x,y\}+(n-x-y-2)-1$ copies of $P_4$.
\end{lem}

\begin{proof}
    As $T$ is connected, every vertex $u\notin D_{x,y}$ is the end-vertex of a $P_4$ such that out of the two end-vertices it is further from $D_{x,y}$. Also, there is a vertex $u^*\notin D_{x,y}$ that is connected to $D_{x,y}$ by an edge, and therefore $u^*$ is an end-vertex of at least $\min\{x,y\}$ copies of $P_4$, and the other end-vertex is on $D_{x,y}$.
\end{proof}

\begin{lem}\label{l:fns}
    Let $T$ be a tree on $n\ge s\ge 4$ vertices. If $T$ contains a copy of $P_s$, then the number of copies of $P_4$ that $T$ contains is at most 
    $$f(n,s):=\begin{cases}
        s-3+2(n-s)+\lfloor \frac{n-s}{2}\rfloor\cdot \lceil \frac{n-s}{2}\rceil, &\quad\text{if}\ s\ge 6,  \\
        \lfloor\frac{n-2}{2}\rfloor\cdot \lceil \frac{n-2}{2}\rceil, &\quad\text{if}\ s=4 ~\text{or}\ s=5.
    \end{cases}$$
\end{lem}

\begin{proof}
    We prove by induction on $n-s$ with the base case $n-s=0,1,2$ being easy or checkable by hand. Also, observe that $f(n,s)\ge f(n,s+1)$, so for the diameter $d$ of $T$ we can always assume $s=d+1$.

    If $d=3$, then $T=D_{x,y}$ with $x+y=n-2$ and the number of copies of $P_4$ is $xy$ which is maximized as claimed by the lemma.

    If $d=4$, and so a longest path in $T$ is $u_1u_2u_3u_4u_5$, then  we denote by $a$ the number of neigbors of $u_3$ and we let $N(u_3)=\{v_1,v_2,\dots,v_a\}$. Denote by $b_i$ the number of neighbors of $v_i$ excluding $u_3$ for $1\le i\le a$. Then $1+a+\sum_{i=1}^ab_i=n$, and the number of copies of $P_4$ in $T$ is $(a-1)\sum_{i=1}^ab_i=(a-1)(n-a-1)$. This is maximized when $a=\lfloor \frac{n}{2}\rfloor$.  

    Finally, we consider $d\ge 5$. Let $P=u_1u_2\dots u_s$ be a longest path in $T$. Note that there are $s-3$ copies of $P_4$ on the path, and for any vertex $v\in V(T)\setminus V(P)$, there are at most 2 copies of $P_4$ starting at $v$ and ending on $P$. We need to show that there are at most $\lfloor \frac{n-s}{2}\rfloor\cdot \lceil \frac{n-s}{2}\rceil$ copies of $P_4$ with both end-vertices in $V(T)\setminus V(P)$.    
    Let $v,v'\in V(T)\setminus V(P)$ be two vertices not on the path that are furthest from each other.
    Suppose first that $v$ is not a leaf in $T$. Then $v$ has to be the only neighbor of a vertex of $P$ in $V(T)\setminus V(P)$. But then all copies of $P_4$ with both end-vertices in $V(T)\setminus V(P)$ lie completely in $V(T)\setminus V(P)$. So, the number of such $P_4$'s, by induction, is at most $f(n-s,s')\le f(n-s,4)=\lfloor \frac{n-s-2}{2}\rfloor \cdot \lceil \frac{n-s-2}{2}\rceil$ as claimed. Now we can assume that both $v,v'$ are leaves of $T$.

        \textsc{Case A} $d_T(v,v')\neq 4$

        Then if $S^r_T(w)$ denotes the set of those vertices of $V(T)\setminus V(P)$ that are at distance exactly $r$ from $w$, then $S^3_T(v)\cap S^3_T(v')=\emptyset$, unless $d_T(v,v')= 6$ in which case there can be at most one vertex in the intersection (or if $D_T(v,v')=4$, but that is \textsc{Case B}). We claim $\min\{|S^3_T(v)|, |S^3_T(v')|\}\le \lfloor \frac{n-s}{2}\rfloor$, which is trivial if the intersection is empty, while it follows from the fact that $v,v'\notin S^3_T(v)\cup S^3_T(v')$ if $d_T(v,v')\neq 3$. We can apply induction to $T\setminus \{v^*\}$, where $v^*\in \{v,v'\}$ with $|S^3_T(v^*)|\le \frac{n-s}{2}$ to obtain that the number of copies of $P_4$ in $T$ is at most $f(n-1,s)+2+\lfloor \frac{n-s}{2}\rfloor=f(n,s)$.

        \textsc{Case B} $d_T(v,v')= 4$

        Then consider the minimal subtree $T'$ of $T$ containing all vertices of $V(T)\setminus V(P)$. As it has diameter 4, there exists a vertex $w$ with $d_{T'}(w,y)\le 2$, for every $y \in V(T')$. Let $a$ denote the number of neighbors of $w$ in $V(T)\setminus V(P)$ and let $b$ denote the number of vertices in $V(T)\setminus V(P)$ that are at distance 2 from $w$. Then $a+b$ is either $n-s$ or $n-s-1$ depending on whether $w\in V(P)$ and the number of copies of $P_4$ with both end-vertices not on $P$ is at most $a\cdot b$ which is maximized when $|a-b|\le 1$.
\end{proof}

Note that the bounds in Lemma \ref{doublestar} are tight as shown by the double-star, the spider with $n/2$ legs of length 2, and their 'extended variants' for larger values of $s$.

\smallskip

Now we are ready to prove Theorem \ref{T2} and Theorem \ref{T3}.

\begin{proof}[Proof of Theorem \ref{T2}]
First, we look at a strategy for Max. Using Lemma \ref{lemaTn} we know that Max can build a star on $\lfloor\frac{n}{2}\rfloor+1$ vertices, regardless of whether he is the first or the second player. Therefore, $\displaystyle{s(n,\# S_k,\cC) \geq \binom{\lfloor\frac{n}{2}\rfloor}{k-1}}$.

Now, we give a strategy for Mini. Based on Lemma \ref{numberstar}, she needs to create at least $\lfloor \frac{n}{2}\rfloor$ vertices of degree at least 2 by the end of the game. If Mini starts, then she can make sure that after her second move, the game graph is a copy of $P_4$. From then on, she always connects a leaf of a component to any other component thus creating a new non-leaf vertex. As there are $n-4$ more moves to be played, by the end of the game she ensures $2+\lfloor \frac{n-4}{2}\rfloor=\lfloor \frac{n}{2}\rfloor$ non-leaf vertices.

If Max starts by claiming an edge $e_1$, then Mini claims an edge $e_2$ with $e_1\cap e_2=\emptyset$. If Max claims $e_3$ with $e_3\cap (e_1\cup e_2)\neq \emptyset$, then Mini creates a copy of $P_5$ with her second move and from then on creates one new non-leaf vertex with each of her moves obtaining a total of $3+\lfloor \frac{n-5}{2}\rfloor\ge \lfloor \frac{n}{2}\rfloor$ non-leaves. Finally, if Max claims $e_3$ to obtain a matching with 3 edges, then Mini joins $e_1,e_2$ to obtain a copy of $P_4$. Now if Max, with his next move, connects a copy of $P_4$ and $e_3$, then he creates a non-leaf, and Mini can always create a new non-leaf obtaining a total of $2+1+\lfloor \frac{n-6}{2}\rfloor\ge \lfloor \frac{n}{2}\rfloor$ non-leaves. Or Max does not connect, but then Mini with her first move can connect two leaves and make them non-leaves and by the end of the game achieves the same number of non-leaves as before.
\end{proof}

\begin{proof}[Proof of Theorem \ref{T3}]
First, we describe a strategy for Max. Using Lemma \ref{lemaTn} we know that Max can build a balanced double star $D_{x,y}$ on $\lfloor\frac{n}{2}\rfloor+1$ vertices, regardless of whether he is the first or the second player.
Therefore, $x+y=\lfloor\frac{n-2}{2}\rfloor$ and thus $x=\displaystyle{\Big\lfloor\frac{\lfloor\frac{n-2}{2}\rfloor}{2}\Big\rfloor, y=\Big\lceil\frac{\lfloor\frac{n-2}{2}\rfloor}{2}\Big\rceil}$. Lemma \ref{doublestar} yields that the number of copies of $P_4$ is at least $xy+\min\{x,y\}+(n-x-y-2)-1$. Plugging in we get that $\displaystyle{s(n,\# P_4,\cC) \geq \frac{n^2}{16}+ \frac{n}{8} - \frac{27}{16}}$.
\medskip

Next, we look at Mini's strategy. Using TreeBuilder's strategy from Lemma \ref{lemaTn} Mini can build a path $P$ on $\lfloor\frac{n}{2}\rfloor+1$ vertices, regardless of whether she is the first or the second player.
%Therefore, the number of vertices that do not belong to $P$ is $\lceil\frac{n}{2}\rceil-1$. The worst scenario for Mini is when all the remaining vertices are in the same balanced double star $S_{\lfloor{x}\rfloor, \lceil{x}\rceil}$ (It is easy to verify that the graph with the largest number of $P_4$’s is a balanced double star). \pb{a proof should be included} There can be at most four vertices that are at the same time in both $P$ and $S_{\lfloor{x}\rfloor, \lceil{x}\rceil}$. Therefore, the double star $S_{\lfloor{x}\rfloor, \lceil{x}\rceil}$ can have at most $\lceil\frac{n}{2}\rceil +1= \lceil\frac{n+2}{2}\rceil$ leaves.  Hence, the number of $P_4$'s at the end of the game is at most $\lfloor{x}\rfloor \lceil{x}\rceil + 2\lceil{x}\rceil + y-1 $, where $x=\displaystyle{\frac{\lceil\frac{n+2}{2}\rceil}{2}}$ $=\Big\lceil\frac{n+2}{4}\Big\rceil \leq \frac{n}{4} + \frac{5}{4}$ and $y=\lfloor\frac{n}{2}\rfloor -3\leq \frac{n}{2}-3$.  
Applying Lemma \ref{l:fns}, we obtain $\displaystyle{s(n,\# P_4,\cC) \leq \frac{n^2}{16}+ \frac{13n}{8} - \frac{35}{16}}$.
\end{proof}

\section{$P_5$-saturation games}\label{p5}

\begin{proof}[Proof of Theorem \ref{T1}]
First, note that the only way to have more than one triangle in a connected component $C$ of a graph that does not have a copy of $P_5$ is if $C$ is a subgraph of $K_4$.

\smallskip 

Now, we give a strategy for Max. 
Assume first that he is the second player. 
If there is at least one isolated vertex after the move of Mini, he follows Stage 1, otherwise, he proceeds to Stage 2.

\medskip 

\textbf{Stage 1}. Max chooses the first of the following three rules that is satisfied, depending on the type of the connected component $C$ Mini played in her last move:

\begin{enumerate}
    \item \label{1} $C$ is a $K_2$.
    Max tries to claim an isolated $K_2$. If that is not possible, he creates a copy of $P_3$ connecting $C$ with one isolated vertex and then proceeds to Stage 2.
    
    \item \label{2} $C$ is a copy of $P_3$.
    Max claims the edge that completes the triangle.
    
    \item \label{3} $C$ has four vertices:
    \begin{enumerate}
        \item\label{3a} If $C$ is a copy of $P_4$, Max claims the edge that completes the copy of $C_4$.
        \item\label{3b} Else, if there is a free edge in that component Max claims it.
        \item\label{3c} Else, if there are at least two isolated vertices Max claims an isolated $K_2$.
        \item\label{3d} Else, if there is one isolated vertex, and on top of that either an isolated edge or an isolated triangle, Max claims the edge that creates, respectively, a copy of $P_3$ or a triangle with one appended edge, which in this proof we will denote by $G_4^1$. Then he proceeds to Stage 2.
        \item\label{3e} Otherwise, Max proceeds to Stage 2.
    \end{enumerate}
     
\end{enumerate}

\textbf {Stage 2}. Max repeatedly plays any legal move till the end of the game.

%Note that the only difference when Max is the first player is that he claims one arbitrary edge on his first move.
\medskip

We now analyze the strategy above. It is clear that Max can follow his strategy in Stage 2. The following claim is an immediate consequence of the above strategy, and so we omit the proof. 

\begin{cl}\label{L1}
After Max has finished his move and remained at Stage 1, every non-trivial connected component of the graph is one of the following: $K_2, C_3, C_4, K_4-e$ or $K_4$.
\end{cl}
%\begin{proof}
%We prove this by using mathematical induction. After the first move of Max, the graph consists of two isolated edges. Now, we suppose that the assertion is true after $k$ rounds. Depending on Mini's $(k+1)$-st move we have the following options:
%\begin{itemize}
%    \item If Max applied rule \ref{1}, there are two more isolated edges in the graph.
%    \item If Max applied rule \ref{2}, there is one isolated edge less and one triangle more than in the previous round.
%    \item If Max applied rule \ref{3a}, there are two isolated edges less and one $C_4$ more than in the previous round. Note that Mini could not make a $P_4$ adding an edge on $P_3$ because of the strategy of Max.
%    \item If Max applied one of the rules \ref{3b}, \ref{3c}, \ref{3d} or \ref{3e}, that means Mini added an edge to one of the following connected components: $C_3, C_4$ or $K_4-e$ and after Max's response we get $K_4-e$, $K_4$ or $K_4 + K_2$ in that order.
%\end{itemize}
%This proves the assertion of the claim.
%\end{proof}
Note that the only connected components on three or four vertices that can be transformed into a connected component on five vertices by adding one edge (without making a copy of $P_5$) are $P_3, P_4, K_{1,3}$ and $G_4^1$. However, following the strategy above, during Stage 1 it is not possible to have any of these components when Mini is to move, as we can see from Claim \ref{L1}.
Therefore, during Stage 1 there is no option for Mini to make a connected component on more than four vertices.

%Note that during Stage 1 Mini can't create a connected component that is a triangle, because of rule \ref{2}. For the same reason, it is not possible to have a connected component $K_{1,3}$ in the graph.   Now it is clear that Stage 1 covers each of Mini's moves.
\smallskip

Now we count the number of triangles at the end of the game.

\smallskip

Observe that at the moment Stage 2 is started, besides components provided by Claim \ref{L1}, there can be at most one of the following three components: $P_3$, $G_4^1$, or an isolated vertex.
\begin{itemize}
    \item In case there is an isolated vertex, that means Max played by rule \ref{3e} and each of the remaining connected components in the graph is one of the following: $C_4$, $K_4-e$ or $K_4$. Hence, at the end of the game, there will be $4 \frac{n-1}{4}= n-1$ triangles.
    
    \item In case there is one $G_4^1$, there are no isolated vertices in the graph. Therefore, this component must become a $K_4$ for the rest of the game. 
    
    \item In case there is one $P_3$, that means that Max played by rule \ref{3d} or rule \ref{1} and there are no more isolated vertices in the graph. Here, Mini can make a connected component on five vertices making the central vertex of this path adjacent to one end of an isolated edge. If that happens, just one edge can be added to this component, the one that closes a triangle.
    
\end{itemize}

Later in Stage 2, the only components that can change the number of vertices are pairs of isolated $K_2$'s that can be transformed into a copy of $K_4$ by the end of the game. Therefore, the worst scenario for Max is that at the end of the game, there is one component on 5 vertices, one isolated edge that cannot be added to any of the components, and all of the remaining components are triangles. This implies $$\displaystyle s_2(n,\#K_3,P_5) \geq \frac{n-4}{3}.$$\\
Note that the same strategy works when Max is the first player, hence $$\displaystyle s_1(n, \#K_3, P_5) \geq \frac{n-4}{3}.$$

\bigskip

Now, we give a strategy for Mini.

%This strategy works both when she is the first player or the second player, hence, w.l.o.g we suppose
Let us first suppose that she is the first player. If there is at least one isolated vertex before the move of Mini, she follows Stage 1, otherwise, she proceeds to Stage 2.

\medskip 

\textbf {Stage 1}. Mini chooses the first of the following three conditions that is satisfied.

\begin{enumerate}
    \item \label{21} There is a connected component on four vertices in the graph.\\
    Then Mini chooses a vertex of maximum degree in this component and connects it to an isolated vertex. (Mini makes a connected component on five vertices.)
    \item \label{22} There is a $K_2$ component in the graph.\\
    Then Mini claims an edge that connects it with an isolated vertex, creating a copy of $P_3$.
    \item \label{23} There are two isolated vertices. \\ 
    Then Mini connects them. 
    \item \label{24} Otherwise, she proceeds to Stage 2.
\end{enumerate}

\textbf {Stage 2}. Mini repeatedly plays any legal move to the end of the game.
\medskip

Now we analyze the given strategy. The following claim is an immediate consequence of the above strategy (and so we omit the proof). 

\begin{cl}\label{L2}
During Stage 1, after Mini has finished her move, the graph consists of the following connected components: at most one $K_2$, some $P_3$'s, some $K_3$'s, and some connected components on at least five vertices.
\end{cl}

Note that each of the connected components on four vertices that cannot be transformed into a bigger one as described in rule \ref{21} (these are the components that belong to the set $\{C_4, K_4-e, K_4\}$) was already a connected component on four vertices before the last edge was added. Now we have a situation similar to the one in Claim \ref{L2}, and we can conclude that neither one of these components can occur when Mini is to move.

At the moment Stage 2 is started, there are two options:
\begin{enumerate}
    \item There is no isolated vertex in the graph.
    \item There is one isolated vertex, but there is neither an isolated edge nor a connected component on four vertices.
\end{enumerate}

If there is one isolated vertex that means that each of the connected components is on three vertices or at least five vertices. The worst case for Mini is when each of the connected components is on three vertices, then it is possible that at most one of them becomes $K_4$, by the end of the game. That implies  $\displaystyle s_2(n,\#K_3,P_5) \leq \frac{n-4}{3} +4$.

Otherwise, if there is no isolated vertex, the graph consists of the following connected components: at most two $K_2$, some $P_3$'s, some $K_3$'s, at most one component from the set \{$K_{1,3}, P_4, G_4^1$\} and some connected components on at least five vertices. Note that if there are two isolated edges then there is no connected component on four vertices and vice versa. Since connected components on three vertices cannot be extended by one vertex, there can be at most one $K_4$ by the end of the game. Therefore, $$\displaystyle s_2(n,\#K_3,P_5) \leq \frac{n-4}{3}+4.$$ As the strategy is the same when Mini is the second player, the same upper bound works for $s_1$.
This concludes the proof of the theorem.
\end{proof}

\begin{proof}[Proof of Theorem \ref{p4p5}]
    First, we obtain a lower bound by describing a strategy of Max. A component is \textit{good} if it contains a copy of $C_4$. At the end of the game, a good component must be a $K_4$. 
    Note that a $K_4$ component contains $12$ copies of $P_4$.
    
    \textsc{Phase I} as long as there exist at least $12\sqrt{n}$ isolated vertices.

    Max makes sure that no triangle is created and after his move, there is no $P_3$-component as follows:
    \begin{itemize}
        \item 
        if Mini extends an edge to a copy of $P_3$, Max extends it to a copy of $P_4$,
        \item 
        if Mini claims an edge $uw$ with $xyuv$ being a path, then Max claims $yz$ with $z$ being an isolated vertex thus creates a $D_{2,2}$, and moves to \textsc{Phase I/B},
        \item 
        if all non-isolated vertices and non-good components are $P_4$'s and edges, then Max prefers to create a copy of $C_4$, if these components form a matching, then Max claims an isolated edge. 
    \end{itemize}

    \textsc{Phase I/B}

    This phase occurs only if a $D_{2,2}$ is created while at least $12\sqrt{n}$ isolated vertices remain. Then only isolated vertices can be added only to the centers of this component. Mini can destroy at most 2 isolated vertices per move, so Max can add at least $4\sqrt{n}$ isolated vertices to the double star and create a $D_{2\sqrt{n},2\sqrt{n}}$ that contains $4n$ copies of $P_4$.

    \textsc{Phase II}

    Max plays arbitrary legal edges. At the beginning of the phase, there are $12\sqrt{n}$ isolated vertices (call them \textit{bad vertices}) any of which can be merged to at most two non-good components. A $P_4$ cannot be merged with an isolated edge, therefore all components without bad vertices will end up being good with the exception of at most one isolated edge. So apart from at most $60\sqrt{n}$ vertices, all others end up in $K_4$'s yielding $3(n-60\sqrt{n})$ copies of $P_4$.

    \medskip

    Next, we consider Mini's strategy to obtain an upper bound on $s(n,\# P_4,P_5)$. A component containing a copy of $C_4$ is still good. Mini makes sure that after her move, no $P_4$-component and no $P_3$-component exists as follows:
    \begin{itemize}
        \item 
        if Max extends an edge to a copy of $P_3$, she closes it to a triangle,
        \item 
        if Max connects two edges to a copy of $P_4$ or if Max connects an isolated vertex to a triangle, she closes the component a good one,
        \item 
        if the non-good, non-isolated vertex components are edges and triangles, then Mini connects two isolated vertices.
    \end{itemize}
    Mini cannot follow this strategy once there remains at most one isolated vertex. In that case, she connects this isolated vertex to an isolated edge or a triangle. If a copy of $P_3$ is created, then this path can be connected to an edge yielding a $D_{1,2}$ containing  2 copies of $P_4$. All other components will either be $K_3$'s or $K_4$'s at the end of the game, so $s(n,\# P_4,P_5)\le 3n$.
\end{proof}

\section{Concluding remarks}

We studied a game analogue of the generalized graph Tur\'an problem, and gave results for several natural choices of $H$ and $F$. Even though all of our lower and upper bounds are reasonably close to each other, in several cases they are not completely tight and it remains an open problem to determine the exact score. 

Furthermore, there are many other choices of $H$ and $F$ for which the $H$-score in the $F$-saturation game remains unexplored. It would be interesting to have more cases of the game where we know the score precisely, and to study its relation to the corresponding generalized Tur\'an number.

As a natural extension of our setup we could drop the restriction of $H$ and $F$ being fixed graphs (or a collection of fixed graphs). Instead, we could allow either of them to be the representatives of any other graph-theoretic structure, say a spanning graph like a spanning tree or a Hamilton cycle.

\end{document}